\documentclass{amsart}

\newcommand{\e}{\varepsilon}
\newcommand{\mchi}{\mathsf{m}\kern-1pt\chi}
\newcommand{\w}{\omega}
\newcommand{\diam}{\mathrm{diam}}
\newcommand{\dist}{\mathrm{dist}}
\newcommand{\IN}{\mathbb N}
\newcommand{\IR}{\mathbb R}
\newcommand{\U}{\mathcal U}

\newcommand{\asdim}{\mathrm{asdim}}
\newcommand{\Ra}{\Rightarrow}
\newcommand{\upa}{{\uparrow\kern-1pt}}
\newcommand{\cof}{\mathrm{cof}}
\newcommand{\coin}{\mathrm{coin}}
\newcommand{\IZ}{\mathbb Z}
\newcommand{\F}{\mathcal F}

\newcommand{\V}{\mathcal V}

\newtheorem{theorem}{Theorem}[section]

\newtheorem{problem}[theorem]{Problem}
\newtheorem{lemma}[theorem]{Lemma}
\newtheorem{claim}[theorem]{Claim}

\title{On character of points in the Higson corona of a metric space}
\author{Taras Banakh, Ostap Chervak, and Lubomyr Zdomskyy}
\address{T.Banakh: Ivan Franko National University of Lviv (Ukraine) and Jan Kochanowski University in Kielce (Poland)}
\email{t.o.banakh@gmail.com}
\address{O.Chervak:  Ivan Franko National University of Lviv, Universytetska 1, 79000, Ukraine}
\email{oschervak@gmail.com}
\address{L.Zdomskyy: Kurt G\"odel Research Center for Mathematical Logic,
University of Vienna, W\"ahringer Str. 25, 1090 Vienna, Austria}
\email{lzdomsky@logic.univie.ac.at}
\thanks{A part of this work was carried out during a visit of
the first named author in June 2011 at the Kurt G\"odel Research Center, supported by
FWF Grant M 1244-N13 of the third named author. The work of the third named author was fully
supported by the same grant. The first author has been partially financed by NCN grant  DEC-2011/01/B/ST1/01439.}
\subjclass{03E17, 03E35, 03E50, 54D35, 54E35, 54F45} 
\keywords{Higson corona, character of a point, ultrafilter number, dominating number}
\begin{document}
\begin{abstract} We prove that for an unbounded metric space $X$, the minimal
character $\mchi(\check X)$ of a point of the Higson corona $\check X$ of $X$
is equal to $\mathfrak u$ if $X$ has asymptotically isolated balls and to
$\max\{\mathfrak u,\mathfrak d\}$ otherwise.
This implies that under $\mathfrak u<\mathfrak d$ a metric space $X$ of
bounded geometry is coarsely equivalent to the Cantor macro-cube $2^{<\IN}$
if and only if $\dim(\check X)=0$ and $\mchi(\check X)=\mathfrak d$. This
contrasts with a result of Protasov saying that under CH the coronas of any
two asymptotically zero-dimensional unbounded metric separable spaces are
homeomorphic.
\end{abstract}
\maketitle

\section{Introduction}

In this paper we shall calculate the smallest character of a point in the
corona $\check X$ of a metric space $X$ and using this information shall
distinguish topologically the Higson coronas of some metric spaces of
asymptotic dimension zero. There are many ways of introducing the Higson
corona of a metric space. We shall follow the approach developed by I.V.Protasov in
\cite{Prot03} and \cite{Prot05}.

For an unbounded metric space $X$, let $\beta X_d$ be the Stone-\v Cech
compactification of the space $X$ endowed with the discrete topology. The
space $\beta X_d$ consists of all ultrafilters on $X$ and carries the compact
Hausdorff topology generated by the sets $\bar A=\{p\in\beta X:A\in p\}$
where $A$ runs over all subsets of $X$. In the space $\beta X_d$ consider the
closed subspace $X^\sharp$ consisting of all ultrafilters that extend the
filter $\F_0=\{X\setminus B:B$ is a bounded subset of $X\}$ of cobounded subsets of $X$.
Two ultrafilters $p,q\in X^\sharp$ are called {\em parallel}
(denoted by $p\parallel q$) if for some positive real number $\e$ we
get $\{B_\e(P):P\in p\}\subset q$ and $\{B_\e(Q):Q\in q\}\subset p$.
Here $B_\e(A)=\{x\in X:d_X(x,A)\le \e\}$ denotes the
$\e$-neighborhood of a subset $A$ of a metric space $(X,d_X)$. The
{\em corona}  $\check X$ of  $X$ is defined as the quotient space
$X^\sharp/_\sim$ of $X^\sharp$ by the smallest closed equivalence
relation $\sim$ on $X^\sharp$ that contains the parallel relation
$\parallel$ on $X^\sharp$. For an ultrafilter $p\in X^\sharp$ by
$\check p\in\check X$ we shall denote its equivalence class in the
corona $\check X$. For a subspace $A\subset X$ we identify the corona
$\check A$ with the subspace $\{\check p:A\in p\in X^\sharp\}$ of $\check X$.

By Proposition 1 of \cite{Prot05}, two ultrafilters $p,q\in X^\sharp$ belong
to the same equivalence class (which means that $\check p=\check q$) if and
only if for any slowly oscillating function $f:X\to [0,1]$ and its Stone-\v
Cech extension $\beta f:\beta X_d\to [0,1]$ we get $\beta f(p)=\beta f(q)$.
This allows us to define the corona $\check X$ of $X$ using
slowly oscillating functions. Let us recall that a function $f:X\to \IR$ is
{\em slowly oscillating} if for any $\e>0$ and any $\delta<\infty$ there is a
bounded subset $B\subset X$ such that for each subset $A\subset X\setminus B$
of diameter $\diam A\le \delta$ the image $f(A)$ has diameter $\diam
f(A)\le\e$. It follows that for a proper metric space $X$ the corona $\check X$
of $X$ coincides with the Higson corona $\nu(X)$ defined in \cite{Roe}.
Let us recall that a metric space $X$ is {\em proper} if each closed bounded
subset of $X$ is compact.

It is known that the coronas $\check X$ and $\check Y$ of two metric spaces
$(X,d_X)$ and $(Y,d_Y)$ are homeomorphic if the metric spaces $X,Y$ are {\em coarsely
equivalent} in the sense that there are two coarse functions $f:X\to Y$ and
$g:Y\to X$ such that $$\max\{\sup_{y\in Y}d_Y(f\circ g(y),y),\sup_{x\in
X}d_X(g\circ f(x),x)\}<\infty.$$ A function $f:X\to Y$ between two metric
spaces $(X,d_X)$ and $(Y,d_Y)$ is called {\em coarse} if for any
$\delta<\infty$ there is $\e<\infty$ such that for any points $x,x'\in X$
with $d_X(x,x')\le\delta$ we get $d_Y(f(x),f(x'))\le\e$.

The topological structure of the corona $\check X$ reflects certain
asymptotic properties of the metric space $X$. In particular, according to \cite{Dra},
\cite[\S5]{BD} for a proper metric space $X$ of finite asymptotic dimension $\asdim(X)$, the
corona $\check X$ has topological dimension $\dim(\check X)=\asdim(X)$. Let us recall that a metric space $X$ has asymptotic
dimension $\asdim(X)\le n$ if for any $\e<\infty$ there is a cover
$\U$ of $X$ such that $\sup_{U\in\U}\diam(U)<\infty$ and each
$\e$-ball $B_\e(x)$, $x\in X$, meets at most $(n+1)$ sets of the
cover $\U$. The finite or infinite number
$$\asdim(X)=\min\{n\in\IN\cup\{\infty\}:\asdim(X)\le n\}$$is called
the {\em asymptotic dimension} of $X$, see \cite{BD}.

It follows that for two proper metric spaces $X,Y$ with different finite
asymptotic dimensions the coronas $\check X$ and $\check Y$ are not
homeomorphic as they have different topological dimensions. On the other
hand, for metric spaces of asymptotic dimension zero I.V.~Protasov \cite{Prot11} proved the following
striking consistency result.

\begin{theorem}[Protasov]\label{t:Prot} Under Continuum Hypothesis the corona
$\check X$ of any asymptotically zero-dimensional unbounded separable metric
space $X$ is homeomorphic to the Stone-\v Cech remainder
$\w^*=\beta\w\setminus\w$ of the countable discrete space $\w$.
\end{theorem}

In a private communication with the first author, I.V.Protasov asked if his Theorem~\ref{t:Prot} remains true in ZFC.
In this paper we shall give a negative answer to this question of Protasov, calculating
the minimal character $\mchi(\check X)$ of
the corona $\check X$ for a metric space $X$.

By definition, the {\em minimal character} $\mchi(X)$ of a topological space
$X$ is the smallest character $\min\limits_{x\in X}\chi(x;X)$ of a point $x$
in $X$, where the {\em character} $\chi(x;X)$ of $x$ in $X$ is equal to the
smallest cardinality of a neighborhood base  at $x$. The minimal character
$\mchi(\w^*)$ of the Stone-\v Cech remainder $\w^*=\beta\w\setminus\w$ is
denoted by $\mathfrak u$ and is one of important small uncountable cardinals,
see \cite{vD}, \cite{Vau}, \cite{Blass}.
Another small uncountable cardinal that will appear in our considerations is
the dominating number $\mathfrak d$, equal to the cofinality of the
partially ordered set $(\w^\w,\le)$, see \cite{vD}, \cite{Vau}, \cite{Blass}.

The cardinals $\mathfrak u$ and $\mathfrak d$ both are equal to the continuum
$\mathfrak c$ under Continuum Hypothesis and more generally under Martin's
Axiom, see \cite{Vau}, \cite{Fremlin}.  On the other hand, the strict inequalities $\mathfrak
u<\mathfrak d$ and $\mathfrak u>\mathfrak d$ also are consistent with ZFC,
see \cite[p.480]{Blass}.

Following \cite{BZ}, we shall say that a metric space $(X,d)$ has {\em
asymptotically isolated balls} if there is $\e<\infty$ such that for any
finite $\delta\ge\e$ there is $x\in X$ such that the $\e$-ball $B_\e(x)$
centered at $x$ coincides with the $\delta$-ball $B_\delta(x)$.


The principal result of this paper is the following theorem that shows that
the conclusion of Protasov's Theorem~\ref{t:Prot} is not true under
$\mathfrak u<\mathfrak d$:

\begin{theorem}\label{main} The corona $\check X$ of an unbounded metric space $X$ has minimal character
$$\mchi(\check X)=\begin{cases}
\mathfrak u&\mbox{if $X$ contains asymptotically isolated balls},\\
\max\{\mathfrak u,\mathfrak d\}&\mbox{otherwise}.
\end{cases}
$$
\end{theorem}

This theorem will be proved in Section~\ref{s:pf-mainth}. Now we shall derive
from Theorem~\ref{main} a corona characterization of the Cantor macro-cube.

The {\em Cantor macro-cube} $2^{<\IN}$ is the metric space
$$2^{<\IN}=\{(x_i)_{i=1}^\infty\in\{0,1\}^\IN:\exists n\in\IN\;\;\forall m\ge
n\;\;x_m=0\}$$endowed with the ultrametric
$$d\big((x_n),(y_n)\big)=\max_{n\in\IN}2^n|x_n-y_n|.$$
By \cite{DZ}, the Cantor macro-cube contains a coarse copy of each
asymptotically zero-dimensional metric space of bounded geometry. Let us
recall that a metric space $X$ has {\em bounded geometry} if there is
$\e<\infty$ such that for every $\delta<\infty$ there is an integer number
$N\in\IN$ such that each $\delta$-ball in $X$ can be covered by $\le N$ balls
of radius $\e$.

The Cantor macro-cube $2^{<\IN}$ is an asymptotic counterpart of the Cantor
cube $2^\w$. According to the classical Brouwer characterization
\cite[7.4]{Ke}, a topological space $X$ is homeomorphic to the Cantor cube
$2^\w$ if and only if $X$ is a zero-dimensional compact metrizable space
without isolated points. A similar characterization holds also for the Cantor
macro-cube \cite{BZ}:
{\em a metric space $X$ is coarsely equivalent to the Cantor macro-cube
$2^{<\IN}$ of and only if
$X$ is an asymptotically zero-dimensonal space of bounded geometry without
asymptotically isolated balls.}

This characterization,  combined with Theorem~\ref{main}, implies the
following ``corona'' characterization of $2^{<\IN}$, which will be proved in
Section~\ref{s2}.

\begin{theorem}\label{t2} Under $\mathfrak u<\mathfrak d$ for a metric space
$X$ of bounded geometry the following conditions are equivalent:
\begin{enumerate}
\item $X$ is coarsely equivalent to $2^{<\IN}$;
\item the corona $\check X$ of $X$ is homeomorphic to the corona of
$2^{<\IN}$;
\item $\dim \check X=0$ and $\mchi(\check X)=\mathfrak d$.
\end{enumerate}
\end{theorem}

Another universal metric space is the {\em Baire macro-space}
$$\w^{<\IN}=\{(x_i)_{i=1}^\infty\in\w^\IN:\exists n\in\IN\;\;\forall m\ge
n\;\;x_m=0\}$$endowed with the ultrametric
$$d\big((x_n),(y_n)\big)=\max(\{0\}\cup\{2^n:x_n\ne y_n\}).$$
The Baire macro-space contains a coarse copy of each separable metric space
of asymptotic dimension zero. Metric spaces that are coarsely equivalent to
the Baire macro-space $\w^{<\IN}$ have been characterized in \cite{BZ2}. By
\cite{Prot11}, under CH the coronas of the metric spaces $2^{<\IN}$ and
$\w^{<\IN}$ are homeomorphic to $\w^*$.

\begin{problem} Can the coronas of the metric spaces $2^{<\IN}$ and
$\w^{<\IN}$ be homeomorphic under the negation of the Continuum Hypothesis?
\end{problem}

\section{Preliminaries}

In this section we collect some information that will be used in the next sections.

By a {\em partial preorder} on a set $P$ we understand any reflexive transitive binary relation $\le$ on $P$. A subset $A\subset P$ of a partially preordered space $(P,\le)$  is called
\begin{itemize}
\item {\em cofinal} in $(P,\le)$ if for each $x\in X$ there is $y\in A$ with $x\le y$;
\item {\em coinitial} in $(P,\le)$ if  for each $x\in X$ there is $y\in A$ with $y\le x$.
\end{itemize}
The smallest cardinality of a cofinal (resp. coinitial) subset of $(P,\le)$ is denoted by $\cof(P)$ (resp. \mbox{$\coin(P)$}) and called the {\em cofinality} (resp. {\em coinitiality}) of $(P,\le)$.

For example, the character $\chi(x,X)$ of a topological space $X$ is equal to the coinitiality of the
 set $\mathcal N_x$ of all neighborhoods of $X$, partially ordered by the inclusion relation $\subset$.

We shall be interested in the cofinality and coinitiality of some function spaces on metric spaces.

A function $f:X\to Y$ between metric spaces is defined to be {\em bounded-to-bounded} if a subset $B\subset X$ is bounded in $X$ if and only if its image $f(B)$ is bounded in $Y$.
We shall be especially interested in bounded-to-bounded functions with values in the space $\w$ of non-negative integers, endowed with the standard Euclidean metric. Observe that a subset $B\subset\w$ is bounded if and only if it is finite. So, a function $\phi:\w\to\w$ is bounded-to-bounded if and only if it is {\em finite-to-one} in the sense that for each $n\in\w$ the preimage $\phi^{-1}(n)$ is finite.

The family of all bounded-to-bounded functions $f:X\to \w$ on a metric space $X$ will be denoted by $\w^{\uparrow X}$. The set $\w^{\upa X}$ carries a natural partial order $\le$ in which $f\le g$ iff $f(x)\le g(x)$ for all $x\in X$.

\begin{lemma}\label{l1} For an unbounded metric space $X$ the partially ordered set $(\w^{\upa X},\le)$ has coinitiality $$\coin(\w^{\upa X})\le\mathfrak d.$$
\end{lemma}

\begin{proof}
Choose any bounded-to-bounded function $\phi:X\to\w$.
By definition of the cardinal $\mathfrak d=\cof(\w^{\upa\w})$, there exits a cofinal set  $\F\subset\w^{\upa\w}$ of cardinality $|\F|=\mathfrak d$.

For each function $f\in\F$, consider the function $\bar f\in\w^{\upa\w}$ defined by
$$\bar f(n)=\max\big(\{0\}\cup\{k\in\w:f(k)\le n\}\big).$$
We claim that the family $\mathcal E=\{\bar f\circ\phi:f\in\F\}$ is coinitial in $\w^{\upa X}$ and hence $\coin(\w^{\upa X})\le|\mathcal E|\le|\F|=\mathfrak d$.

Indeed, take any function $g\in\w^{\upa X}$ and consider the function $\tilde g\in\w^{\upa \w}$ defined by
$$\tilde g(n)=\min g\big(\phi^{-1}([n,\infty))\big)\mbox{ \ for $n\in\w$}.$$
Next, consider the function $\tilde f\in\w^{\upa\w}$ defined by $$\tilde f(k)=\min(\tilde g^{-1}([k+1,\infty))\mbox{ for $k\in\w$}$$and choose any function $f\in\F$ with $\tilde f\le f$.

We claim that $\bar f\circ\phi\le g$. Take any point $x\in X$ and consider the number $n=\phi(x)$. Then $\tilde g(n)\le g(x)$. Let $k=\tilde g(n)$ and observe that $$n\le \max \tilde g^{-1}(k)<\min \tilde g^{-1}([k+1,\infty))=\tilde f(k)\le f(k).$$
Now the defintion of $\bar f(n)$ implies that
$$\bar f\circ \phi(x)=\bar f(n)\le k=\tilde g(n)\le g(x).$$
\end{proof}

Now consider the space $\w^{\upa\w}$ of bounded-to-bounded (=finite-to-one) functions on $\w$.
Besides the coinitiality of the partial order $\le$ on $\w^{\upa\w}$ we shall be interested in the coinitiality of $\w^{\upa\w}$ endowed with the linear preorder $\le_\U$ generated by an ultrafilter $\U\in\w^*$.
For two functions $f,g\in\w^{\upa\w}$ we write $f\le_\U g$ if the set $\{n\in\w:f(n)\le g(x)\}$ belongs to the ultrafilter $\U$. Following \cite{BZbook}, we denote by $\mathfrak q(\U)=\coin(\w^{\upa\w},\le_\U)$ and $\mathfrak d(\U)=\cof(\w^{\upa\w},\le_\U)$ the coinitiality and the cofinality of the linearly preordered space $(\w^{\upa\w},\le_\U)$. It is clear that $\max\{\mathfrak q(\U),\mathfrak d(\U)\}\le\mathfrak d$. In \cite{Canjar} M.Canjar constructed a ZFC-example of an ultrafilter $\U\in\w^*$ with $\mathfrak q(\U)=\mathfrak d(\U)=\mathrm{cf}(\mathfrak d)$, which can be consistently smaller than $\mathfrak d$.

The following lemma can be proved by analogy with Theorem 16 of \cite{Bla86}, see also Theorem 9.4.6 of \cite{BZbook} or \cite[pp.82,85]{BZsurv}. In this Lemma $\chi(\U)$ denotes the character of an ultrafilter $\U\in\w^*$ in the Stone-\v Cech compactification $\beta(\w)$ of $\w$.

\begin{lemma}\label{l2} Any ultrafilter $\U\in\w^*$ with character $\chi(\U)<\mathfrak d$ has
$\mathfrak q(\U)=\mathfrak d(\U)=\mathfrak d$. Consequently, $$\max\{\chi(\U),\mathfrak q(\U)\}=\max\{\chi(\U),\mathfrak d(\U)\}=\max\{\chi(\U),\mathfrak d\}\ge\max\{\mathfrak u,\mathfrak d\}$$for any ultrafilter $\U\in\w^*$.
\end{lemma}

We shall need to generalize the definition of a ball $B_\e(x)$ to allow the radius to take a function value. Namely, for a function $f:X\to[0,\infty)$ defined on a metric space $X$, a point $x\in X$ and a subset $A\subset X$, let $B(x,f)=\{y\in X:d(y,x)\le f(x)\}=B_{f(x)}(x)$ and
$$B(A,f)=\bigcup_{a\in A}B(a,f).$$ The set $B(A,f)$ is called the {\em $f$-neighborhood} of $A$ in $X$.
Sometimes for a real number $\e\ge0$ we shall use the notation $B(x,\e)$ instead of $B_\e(x)$ identifying $\e$ with the constant function $\e:X\to\{\e\}\subset[0,\infty)$.

For a set $A\subset X$ and a function $f:X\to[0,\infty)$, the $f$-neighborhood
$B(A,f)\subset X$ determines the closed-and-open set $\bar B(A,f)=\{p\in X^\sharp:B(A,f)\in p\}$ in the compact Hausdorff space $X^\sharp\subset\beta X$ and the closed subset $\check B(A,f)=\{\check p:p\in\bar B(A,f)\}$ in the corona $\check X$ of $X$.

We shall use the following description of the topology $\check X$, mentioned in \cite{Prot11}.

\begin{lemma}\label{l3} For each ultrafilter $p\in X^\sharp$ the family $$\{\check B(P,f):P\in p,\;f\in \w^{\upa X}\}$$ is a base of closed neighborhoods of $\check p$ in $\check X$.
\end{lemma}

This lemma implies an easy criterion for recognizing ultrafilters $p,q\in X^\sharp$ with different images $\check p$, $\check q$. We say that two subsets $P,Q$ of a metric space $(X,d)$ are {\em asymptotically disjoint} if for each real number $\e>0$ the intersection $B(P,\e)\cap B(Q,\e)$ is bounded in $X$. This is equivalent to the existence of a bounded-to-bounded function $f\in\w^{\upa X}$ such that the intersection $B(P,f)\cap B(Q,f)$ is bounded.

The following fact was proved by I.V.Protasov in Lemma 4.2 of \cite{Prot03}.

\begin{lemma}\label{l4} For an unbounded metric space $X$ two ultrafilters $p,q\in X^\sharp$ have distinct images $\check p\ne\check q$ in the corona $\check X$ if and only if there are two asymptotically disjoint sets $P,Q\subset X$ such that $P\in p$ and $Q\in q$.
\end{lemma}

\begin{proof} If $\check p\ne\check q$, then we can choose two disjoint neighborhoods $O(\check p)$ and $O(\check q)$ of the points $\check p$, $\check q$ in the corona $\check X$. By Lemma~\ref{l3}, we can assume that these neighborhoods are of the form $O(\check p)=\check B(P,f)$, $O(\check q)=\check B(Q,f)$ for some sets $P\in p$, $Q\in q$ and some bounded-to-bounded function $f\in\w^{\upa X}$. To see that the sets $P,Q$ are asymptotically disjoint, it suffices to check that the intersection $B(P,f)\cap B(Q,f)$ is bounded. Assuming the opposite, we could find an ultrafilter $r\in X^\sharp$ containing $B(P,f)\cap B(Q,f)$. Then $\check r\in \check B(P,f)\cap\check B(Q,f)=O(\check p)\cap O(\check q)$, which is not possible as the sets $O(\check p)$ and $O(\check q)$ are disjoint.
This proves the ``only if'' part of the lemma.

To prove the ``if'' part, assume that two ultrafilters $p,q\in X^\sharp$ contain asymptotically disjoint sets $P\in p$, $Q\in q$. Choose a bounded-to-bounded function $f\in\w^{\upa X}$ such that $B(P,f)\cap B(Q,f)$ is bounded. Then $\check B(P,f)$ and $\check B(Q,f)$ are two disjoint neighborhoods of the points $\check p$ and $\check q$, which implies that $\check p\ne\check q$.
\end{proof}

A subset $A$ of a metric space $X$ is called {\em asymptotically isolated} if $A$ is asymptotically disjoint from its complement $X\setminus A$. This happens if and only if $B(A,f)=A$ for some bounded-to-bounded function $f\in\w^{\upa X}$. For a subset $A\subset X$ let $\check A=\{\check p:A\in p\in X^\sharp\}$.

\begin{lemma}\label{l5a} A subset $\U\subset \check X$ is closed-and-open in the corona $\check X$ if and only if $\U=\check U$ for some asymptotically isolated subset $U\subset X$.
\end{lemma}

\begin{proof} Assume that $\U=\check U$ for some asymptotically isolated subset $U\subset X$. Then $B(U,f)=U$ for some bounded-to-bounded function $f\in \w^{\upa X}$. It follows from Lemma~\ref{l3} that for each ultrafilter $p\in X^\sharp$ with $\check p\in \check U$ the set $\check B(U,f)=\check U$ is a neighborhood of $\check p$, which means that $\check U=\U$ is open in $\check X$. The set $\check U=\U$ is closed being a continuous image of the compact subset $\bar U=\{p\in X^\sharp:U\in p\}$.
\smallskip

Now assume that a subset $\U\subset \check X$ is closed-and-open in $\check X$. Fix any point $x_0$ in the metric space $X$. Since the set $\U$ is open in $\check X$, for each ultrafilter $p\in X^\sharp$ with $\check p\in\U$, there is a set $P_p\in p$ and a bounded-to-bounded function $f_p\in\w^{\upa X}$ such that $\check B(P_p,3f_p)\subset \U$. Replacing $f_p$ by a smaller function, if necessary, we can assume that
$B(B(x,f_p),f_p)\subset B(x,3f_p)$ and $f_p(x)\le\frac12d(x,x_0)$ for each point $x\in X$.

By the compactness of $\U$, the cover $\{\check B(P_p,f_p):p\in X^\sharp,\;\;\check p\in\U\}$ has a finite subcover $\{\check B(P_p,f_p):p\in F\}$ where $F\subset X^\sharp$ is a finite set. Now consider the set $U=\bigcup_{p\in F}B(P_p,f_p)$ and observe that $\check U=\bigcup_{p\in F}\check B(P_p,f_p)=\U$. Let $f=\min\{f_p:p\in F\}$ and observe that
$$\check B(U,f)=\bigcup_{p\in F}\bigcup_{x\in P_p}B(B(x,f_p),f)\subset \bigcup_{p\in F}\bigcup_{x\in P_p}B(x,3f_p)=\bigcup_{p\in F}B(P_p,3f_p)$$ and hence $$\U=\check U\subset \check B(U,f)\subset\bigcup_{p\in F}\check B(P_p,3f_p)\subset\U.$$ The equality $\check U=\check B(U,f)$ implies that the set $B(U,f)\setminus U$ is bounded. It follows from $f(x)\le \frac12d(x,x_0)$, $x\in X$, that the set $D=\{x\in X: B(x,f)\cap (B(U,f)\setminus U)\ne \emptyset\}$ is bounded in $X$.
Now define a bounded-to-bounded function $f_0\in\w^{\upa X}$ letting $f_0|D\equiv 0$ and $f_0|X\setminus D=f|X\setminus D$.

We claim that $B(U,f_0)=U$. Assuming the opposite, find a point $x\in B(U,f_0)\setminus U$ and a point $u\in U$ with $x\in B(u,f_0)$. The definition of the set $D$ guarantees that $u\in D$ and hence $f_0(u)=0$ and $x=u\in U$, which is a contradiction.
The equality $U=B(U,f_0)$ witnesses that the set $U$ with $\check U=\U$ is asymptotically isolated.
\end{proof}

Balls $B(x,f)$ with function radius $f\in\w^{\upa X}$ can be used
 to prove the following characterization of coarse maps in spirit of uniform continuity.

\begin{lemma} A bounded-to-bounded function $f:X\to Y$ between metric spaces is coarse if and only if

$\forall \e\in\w^{\upa Y}\;\exists\delta\in\w^{\upa X}\;\forall x\in X\;\;f(B(x,\delta))\subset B(f(x),\e)$.
\end{lemma}

\begin{proof} To prove the ``only if'' part, assume that the bounded-to-bounded function $f:X\to Y$ is coarse. In this case there is an increasing function $\xi:\w\to\w$ such that for any $n\in\w$ and points $x,x'\in X$ with $d_X(x,x')\le n$ we get $d_Y(f(x),f(x'))\le \xi(n)$. Consider the bounded-to-bounded function $\zeta:\w\to\w$, $\zeta:m\mapsto\max\{n\in\w:\xi(n)\le m\}$ and observe that $\xi\circ\zeta(m)\le m$ for each $m\in\w$.

Given any bounded-to-bounded function $\e\in\w^{\upa Y}$, consider the bounded-to-bounded function  $\delta:X\to \w$, $\delta(x)=\zeta\circ\e\circ f(x)$, and observe that it has the required property: $f(B(x,\delta)\subset B(f(x),\e)$ for all $x\in X$.

To prove the ``if'' part, choose any bounded-to-bounded function $\e\in{\upa X}$ and assume that there exists $\delta\in\w^{\upa X}$ such that $f(B(x,\delta))\subset B(f(x),\e)$ for all $x\in X$.
To show that $f$ is coarse, for each real number $r$ we need to find a real number $R$ such that $f(B_r(x))\subset B(f(x),R)$.  Since the function $\delta:X\to\w$ is bounded-to-bounded, the set
$\Delta=\delta^{-1}([0,r))$ is bounded in $X$ and so is its $r$-neighborhood
$B_r(\Delta)=\bigcup_{x\in\Delta}B(x,r)$. Since the functions $f$ and $\e$ are bounded-to-bounded, the set $f(B_r(\Delta))$ is bounded in $Y$ and $\e\circ f(B_r(\Delta))$ is bounded in $\w$.
It can be shown that the number
$$R=\max\big\{\e(r),\diam\big(\e\circ f(B_r(\Delta))\big)\big\}$$has the required property: $f(B_r(x))\subset B_R(f(x))$ for each $x\in X$.
\end{proof}

A function $\phi:X\to Y$ between two metric spaces is called {\em boundedly oscillating} if there is a real number $D$ such that for any real number $\e$ there is a bounded set $B\subset X$ such that for each point $x\in X\setminus B$ the set $\phi(B_\e(x))$ has diameter $\diam\, \phi(B_\e(x))\le D$.
It is clear that each slowly oscillating function is boundedly oscillating.

The following characterization of boundedly oscillating functions easily follows from the definition.

\begin{lemma}\label{l:bo} A function $\phi:X\to Y$ between metric spaces is boundedly oscillating if and only if there is a bounded-to-bounded function $\e\in\w^{\upa X}$ such that $\sup_{x\in X}\diam \,\phi(B(x,\e))<\infty$.
\end{lemma}

Using Lemma~\ref{l:bo} it is quite easy to construct boundedly oscillating functions $f:X\to\w$ with values in $\w$.

\begin{lemma} For each metric space $X$ there is a boundedly oscillating bounded-to-bounded function $\phi:X\to\w$.
\end{lemma}

\begin{proof} Fix any point $x_0\in X$ and choose an increasing sequence of real numbers $(r_n)_{n\in\w}$ such that $r_{0}<0$ and $\lim_{n\to\infty}r_{n+1}-r_n=\infty$. Then the function $\phi:X\to\w$ defined by $\phi^{-1}(n)=B_{r_{n+1}}(x_0)\setminus B_{r_n}(x_0)$ for $n\in\w$ is boundedly oscillating and bounded-to-bounded.
\end{proof}

\begin{lemma}\label{l:bobf} For any boundedly oscillating bounded-to-bounded function $\phi:X\to\w$ on an unbounded metric space there is a bounded-to-bounded function $\tilde \e\in\w^{\upa \w}$ such that $\sup_{x\in X}\diam\, \phi(B(x,\tilde\e\circ \phi))<\infty$.
\end{lemma}

\begin{proof} By Lemma~\ref{l:bo}, there is a bounded-to-bounded function $\e\in\w^{\upa X}$ such that
$$ D=\sup_{x\in X}\diam\, \phi(B(x,\e))<\infty.$$ Since the map $\phi:X\to\w$ is bounded-to-bounded, there is a bounded-to-bounded function $\tilde \e\in\w^{\upa\w}$ such that $\tilde\e\circ\phi\le\e$. Such function $\tilde \e$ can be defined by the formula
$$\tilde \e(n)=\min\e(\phi^{-1}([n,\infty))\mbox{ \ for $n\in\w$}.$$
The inequality $\tilde\e\circ\phi\le\e$ implies
$$\sup_{x\in X}\diam\, \phi(B(x,\tilde\e\circ \phi))\le \sup_{x\in X}\diam\, \phi(B(x,\e)) <\infty.$$
\end{proof}

Observe that for a bounded-to-bounded function $\phi:X\to\w$ defined on an unbounded metric space $X$ and an ultrafilter $p\in X^\sharp$ its image $\beta \phi(p)=\{A\subset \w:\phi^{-1}(A)\in p\}$ lies in the set $\w^\sharp=\w^*\subset\beta\w$. To shorten notations, we shall denote the image $\beta \phi(p)$ of the ultrafilter $p$ by $\phi(p)$.

\section{Dimension of the corona}

By \cite{Dra}, for each proper metric space $X$ of finite asymptotic dimension $\asdim(X)$ the corona $\check X$ has topological dimension $\dim(\check X)=\asdim(X)$. However it is not known if the asymptotic dimension $\asdim(X)$ is finite provided that the topological dimension $\dim(\check X)$ of the corona $\check X$ is finite (cf. \cite[\S5]{BD}).
In this section we give an affirmative answer to this problem for metric spaces $X$ with zero-dimensional corona. We shall apply a characterization of asymptotic dimension zero in terms of $\e$-chains.

Let $\e\ge 0$ be a real number. By an {\em $\e$-chain} in a metric space $(X,d)$ we understand any sequence of points $x_0,\dots,x_n$ of $X$ such that $d(x_{i-1},x_i)\le \e$ for all positive $i\le n$.
For a point $x\in X$ its {\em $\e$-component} $C_\e(x)$ is the set of all points $y\in X$, which can be linked with $x$ by an $\e$-chain $x=x_0,x_1,\dots,x_n=y$.

\begin{theorem}\label{t:dim} For an unbounded metric space $X$ the following conditions are equivalent:
\begin{enumerate}
\item $X$ has asymptotic dimension zero;
\item $\sup_{x\in X}\diam C_\e(x)<\infty$ for each $\e<\infty$;
\item the corona $\check X$ has topological dimension zero.
\end{enumerate}
\end{theorem}

\begin{proof}
$(1)\Ra(2)$. Assume that $X$ has asymptotic dimension zero. Then for each $\e<\infty$ there is a cover $\U$ of $X$ such that $\sup_{U\in\U}\diam(U)<\infty$ and each $\e$-ball $B_\e(x)$, $x\in X$, meets a unique set $U\in\U$. Then for each point $x\in X$ its $\e$-component $C_\e(x)$ lies in a unique set $U\in\U$, which implies that $$\sup_{x\in X}\diam C_\e(x)\le \sup_{U\in\U}\diam(U)<\infty.$$
\smallskip

The implication $(2)\Ra(1)$ trivially follows from the fact that for each $\e<\infty$, $\U=\{C_\e(x):x\in X\}$ is a disjoint cover of $X$ such that each $\e$-ball $B_\e(x)$, $x\in X$, meets a unique set $U\in\U$ (which is equal to $C_\e(x)$).
\smallskip

$(2)\Ra(3)$ Assume that for each $\e\ge 0$ the number $\gamma(\e)=\sup_{x\in X}\diam C_\e(x)$ is finite. Since the space $X$ is unbounded, the function $\gamma:[0,\infty)\to[0,\infty)$ is bounded-to-bounded.

To show that the corona $\check X$ of $X$ has topological dimension zero, fix any ultrafilter $p\in X^\sharp$ and a neighborhood $U\subset \check X$ of its equivalence class $\check p$. By Lemma~\ref{l3}, we can assume that $U$ is of the form $U=\check B(P,f)$ where $P\in p$ and $f:X\to \w$ is a bounded-to-bounded function.

Fix any point $x_0\in X$ and put $\|x\|=d(x,x_0)$ for any point $x\in X$. Replacing $f$ by a smaller function, if necessary, we can assume that $f(x)\le \frac12\|x\|$. This condition guarantees that for any point $x\in X$ and $y\in B(x,f)$ we get
$$\|y\|=d(y,x_0)\le d(y,x)+d(x,x_0)\le f(x)+d(x,x_0)\le \frac12\|x\|+\|x\|=\frac32\|x\|$$and
$$\|x\|=d(x,x_0)\le d(x,y)+d(y,x_0)\le f(x)+\|y\|\le \frac12\|x\|+\|y\|,$$ which implies
$\frac12\|x\|\le \|y\|$. Consequently,
\begin{equation}\label{eq:ostap1}
\frac23\|y\|\le\|x\|\le 2\|y\|\mbox{ \ for any points $x\in X$ and $y\in B(x,f)$}.
\end{equation}

Consider the bounded-to-bounded function $\e:X\to[0,\infty)$ defined by
$$\e(x)=\frac12\sup\{\e\ge0:\gamma(\e)\le f(x)\}\mbox{ \ for \ }x\in X,$$
and observe that $C_{\e(x)}(x)\subset B(x,f(x))$ for all $x\in X$. Using the inequalities (\ref{eq:ostap1}), one can check that the function $$\delta:X\to [0,\infty),\;\;\;\delta:x\mapsto
\inf \{\e(y):x\in C_{\e(y)}(y)\},$$
is bounded-to-bounded. So, we can choose a bounded-to-bounded function $\tilde f:X\to\w$ such that $\tilde f(x)\le\delta(x)$ for all $x\in X$.

The choice of the function $\e$ guarantees that the set $\tilde P=\bigcup_{x\in P}C_{\e(x)}(x)$ belongs to the ultrafilter $p$ and lies in the $f$-neighborhood $B(P,f)$ of the set $P$. Moreover, $B(\tilde P,\tilde f)=\tilde P$. Indeed, for each point $x\in \tilde P$ we can find a point $y\in P$ with $x\in C_{\e(y)}(y)$. Then definition of the function $\delta$ guarantees that $\tilde f(x)\le\delta(x)\le \e(y)$, which implies that $B(x,\tilde f)\subset C_{\e(y)}(y)\subset \tilde P$. So, $B(\tilde P,\tilde f)=\tilde P$, which implies that $\check B(\tilde P,\tilde f)\subset\check B(P,f)$ is a closed-and-open neighborhood of $\check p$ in $\check X$.
\smallskip

$(3)\Ra(2)$ To derive a contradiction, assume that $\dim(\check X)=0$ but there is $\e<\infty$ such that $\sup_{x\in X}\diam C_\e(x)=\infty$. For two subsets $A,B\subset X$ put $\dist(A,B)=\inf\{d(a,b):a\in A,\; b\in B\}$.
Fix any point $\theta\in X$.

\begin{claim}\label{cl3.2} There is a sequence $(C_n)_{n\in\w}$ of bounded $\e$-connected subsets of $X$ such that
$\diam C_n>n$ and $\dist(C_n,C_{<n})\ge n$ where $C_{<n}=B_n(\theta)\cup \bigcup_{k<n}C_k$.
\end{claim}

\begin{proof} The sets $C_n$, $n\in\w$, will be constructed by induction. Assume that for some number $n\in\w$ bounded $\e$-connected sets $C_0,\dots,C_{n-1}$ have been constructed. Consider the bounded set $C_{<n}=B_n(\theta)\cup\bigcup_{k<n}C_k$ and its $n$-neighborhood $B=B_n(C_{<n})=\bigcup_{c\in C_{<n}}B_n(c)$.

Now we consider two cases.

(i) $D=\sup_{x\in B}\diam C_\e(x)<\infty$. Since $\sup_{x\in X}C_\e(x)=\infty$, we can choose a point $x\in X$ such that $\diam C_\e(x)>2\max\{n,D\}$. It follows that $x\notin B$ and moreover, $C_\e(x)\cap B=\emptyset$ (in the opposite case, for a point $y\in B\cap C_\e(x)$, its $\e$-connected component $C_\e(y)=C_\e(x)$ has diameter $\diam C_\e(y)>2D\ge D$, which contradicts the definition of $D$). So, $C_\e(x)\cap B=\emptyset$.

Since $\diam C_\e(x)>2n$, we can choose a point $y\in C_\e(x)$ such that $d(y,x)>n$. By the definition of the set $C_\e(x)$, the points $x,y\in C_\e(x)$ can be linked by an $\e$-chain $x=x_0,\dots,x_m=y$. Then $C_n=\{x_0,\dots,x_m\}$ is a required bounded $\e$-connected subset of $X$ that has diameter $\diam C_n\ge d(x,y)>n$ and $$\dist(C_n,C_{<n})\ge \dist(C_\e(x),C_{<n})\ge \dist(X\setminus B,C_{<n})\ge n.$$

(ii) The second case happens when $\sup_{x\in B}\diam C_\e(x)=\infty$. In this case we can choose a point $y\in B$ such that $\diam C_\e(y)>2(\diam(B)+n+\e)$. Then there is a point $x\in C_\e(y)$ with $d(x,y)>\diam(B)+n+\e$, which can be linked with $y$ by an $\e$-chain  $x=x_0,\dots,x_m=y$. Since $d(x_0,x_m)=d(x,y)>\diam(B)+n+\e$, we can choose the smallest number $k\le m$ such that $d(x_0,x_k)>n$.
Then $d(x_0,x_i)\le n$ for every $i<k$ and hence $$d(x_i,B)\ge d(x_i,y)-\diam(B)\ge d(x_0,y)-d(x_0,x_i)-\diam(B)>\diam(B)+n+\e-n-\diam(B)=\e.$$
Also $d(x_k,B)\ge d(x_{k-1},B)-d(x_{k-1},x_k)>\e-\e=0$. Consequently, the bounded $\e$-connected set $C_n=\{x_0,\dots,x_k\}$ has diameter $\diam(C_n)\ge d(x_0,x_k)>n$ and is disjoint with the set $B=B_n(C_{<n})$, which implies that $\dist(C_n,C_{<n})\ge n$. This completes the inductive construction.
\end{proof}

Claim~\ref{cl3.2} yields a sequence $(C_n)_{n\in\w}$ of $\e$-connected sets such that $\diam (C_n)>n$ and $\dist(C_n,C_{<n})\ge n$ for each $n\in\w$. For every $n\in\w$ choose two points $x_n,y_n\in C_n$ on distance $d(x_n,y_n)>n$. The choice of the sets $C_n\subset X\setminus B_n(\theta)$, $n>0$, implies that the sequences $\vec x=(x_n)_{n\in\w}$ and $\vec y=(y_n)_{n\in\w}$ tend to infinity and the sets $P=\{x_n\}_{n\in\w}$ and $Q=\{y_n\}_{n\in\w}$ are unbounded and asymptotically disjoint.

The sequences $\vec x$ and $\vec y$ can be thought as functions $\vec x:\w\to X$ and $\vec y:\w\to Y$ and so have the Stone-\v Cech extensions $\beta \vec x:\beta\w\to\beta X_d$ and $\beta\vec y:\beta\w\to\beta X_d$. Since the sequences $\vec x$ and $\vec y$ tend to infinity, $\beta\vec x(\w^*)\cup\beta\vec y(\w^*)\subset X^\sharp$. Take any free ultrafilter $\F\in\w^*$ and consider its images $p=\beta\vec x(\F)\in X^\sharp$ and $q=\beta\vec y(\F)\in X^\sharp$. Since the sets $\vec x(\w)\in p$ and $\vec y(\w)\in q$ are asymptotically disjoint, $\check p\ne\check q$ according to Lemma~\ref{l4}.

Since the space $\check X$ has topological dimension zero, there are disjoint open-and-closed sets  $\U,\V\subset \check X$ such that $\check p\in\U$ and $\check q\in\V$. By Lemma~\ref{l5a} there are  asymptotically isolated sets $U,V\subset X$ such that $\U=\check U$ and $\V=\check V$.
Since $U,V$ are asymptotically isolated in $X$, there is a bounded-to-bounded function $f\in\w^{\upa X}$ such that $B(U,f)=U$ and $B(V,f)=V$.

It follows from $\check U\cap \check V=\U\cap\V=\emptyset$ that the intersection $U\cap V$ is bounded.
Choose $n\in\w$ so large that
\begin{itemize}
\item the $n$-ball $B_n(\theta)$ contains the bounded set $U\cap V$, and
\item $f(x)>\e$ for each $x\in X\setminus B_n(\theta)$.
\end{itemize}

It follows from $\check p\in \U=\check U$ and $\check q\in\V=\check V$ that $U\in p=\beta\vec x(\F)$ and $V\in q=\vec y(\F)$. Consider the (infinite) set $F=\vec  x^{-1}(U\setminus B_n(\theta))\cap\vec y^{-1}(V\setminus B_n(\theta))\in\F$. Choose any number $m\in F$ with $m>n$ and consider the $\e$-connected set $C_m$. By Claim~\ref{cl3.2}, $C_m\cap B_n(\theta)\subset C_m\cap B_m(\theta)=\emptyset$. Choose an $\e$-chain $x_m=z_0,\dots,z_k=y_m$ linking the points $x_m$ in $y_m$ of the set $C_m$.
Observe that $z_0=x_m\in U\setminus B_n(\theta)$ and $z_k=y_m\in V\setminus B_n(\theta)\subset X\setminus U$. So, the largest number $l\le k$ such that $z_l\in U$ is not equal to $k$.
It follows from $z_l\in C_m\subset X\setminus B_m(\theta)\subset X\setminus B_n(\theta)$ and the choice of the number $n$ that $f(z_l)>\e$.

Then $z_{l+1}\in B_\e(z_l)\subset B_{f(z_l)}(z_l)=B(z_l,f)\subset B(U,f)=U$, which contradicts the definition of $l$.
\end{proof}

\section{Evaluating the character of a point in the corona}

In this section, for an unbounded metric space $(X,d)$ and an ultrafilter $p\in X^\sharp$ we shall evaluate the character $\chi(\check p,\check X)$ of the point $\check p$ in the corona $\check X$ of $X$.

First we derive an upper bound on $\chi(\check p,\check X)$ from Lemmas~\ref{l1} and \ref{l3}.

\begin{lemma}\label{l5} For each ultrafilter $p\in X^\sharp$ the point $\check p\in\check X$ has character
$$\chi(\check p,\check X)\le \max\{\chi(p,X^\sharp),\mathfrak d\}.$$
\end{lemma}

\begin{proof} Let $\kappa= \max\{\chi(p,X^\sharp),\mathfrak d\}$. Since $\chi(p,X^\sharp)\le\kappa$, there is a family $\mathcal P\subset p$ of cardinality $|\mathcal P|=\chi(p,X^\sharp)\le\kappa$ such that for each set $P\in p$ there is a set $Q\in\mathcal P$ with $\bar Q\subset\bar P$, where $\bar Q=\{q\in X^\sharp:Q\in q\}$. We claim that the complement $Q\setminus P$ is bounded. In the other case, there is an ultrafilter $q\in X^\sharp$ such that $Q\setminus P\in p$. Then $q\in \bar Q\setminus\bar P$, which is a contradiction.

Fix any point $\theta\in X$ and consider the enriched family $\mathcal P'=\{P\setminus B_n(\theta):P\in\mathcal P,\;n\in\w\}\subset p$. It is clear that $|\mathcal P'|\le\aleph_0\cdot|\mathcal P|\le\kappa$ and for each set $P\in p$ there is a set $P'\in\mathcal P'$ with $P'\subset P$.

By Lemma~\ref{l1}, the partially ordered set $(\w^{\upa\w},\le)$ has coinitiality $\coin(\w^{\upa X})\le \mathfrak d$. So, we can find a coinitial set $\F\subset \w^{\upa X}$ of cardinality $|\mathcal F|\le\mathfrak d$.

It follows that for each set $P\in p$ and a function $g\in\w^{\upa X}$ there is a set $P'\in\mathcal P'$ and a function $f\in\F$ such that $P'\subset P$ and $f\le g$. Then $p\in \bar B(P',f)\subset \bar B(P,g)$ and hence $\check p\in\check B(P',f)\subset \check B(P,g)$, which implies that $\{\check B(P,f):P\in\mathcal P',\;f\in\F\}$ is a neighborhood base at $\check p$ and $\chi(\check p,\check X)\le |\mathcal P'|\cdot|\F|\le\kappa$.
\end{proof}

\begin{lemma}\label{l6} If $\phi:X\to\w$ is a boundedly oscillating bounded-to-bounded function, then
for each ultrafilter $p\in X^\sharp$ the point $\check p\in\check X$ has character
$$\chi(\check p,\check X)\ge \chi(\phi(p),\w^*).$$
\end{lemma}

\begin{proof} Assume conversely that the cardinal $\kappa=\chi(\check p,\check X)$ is smaller that $\chi(\phi(p),\w^*)$. Using Lemma~\ref{l3}, choose a transfinite sequence of pairs $(P_\alpha,f_\alpha)\in p\times\w^{\upa X}$, $\alpha<\kappa$, such that for each pair $(P,f)\in p\times\w^{\upa X}$ there is an ordinal $\alpha<\kappa$ with $\check B(P_\alpha,f_\alpha)\subset \check B(P,f)$.

By Lemma~\ref{l:bobf}, there is a function
$\tilde f\in\w^{\upa \w}$ such that $$D=\sup_{x\in X}\diam\,\phi\big(B(x,\tilde f\circ \phi)\big)<\infty.$$ Let $f=\tilde f\circ\phi$ and choose any natural number $l>2D$.

Since $\phi(p)$ is an ultrafilter on $\w=\bigcup_{i=0}^{l-1}l\w+i$, there is a non-negative integer number $i<d$ such that the set $l\w+i=\{ln+i:n\in\w\}$ belongs to $\phi(p)$.

For every $\alpha<\kappa$ consider the set $Q_\alpha=(l\w+i)\cap \phi(P_\alpha)\in\phi(p)$. Since the family $\{Q_\alpha\}_{\alpha<\kappa}$ has cardinality $\le\kappa<\chi(\phi(p),\w^*)$, there exists a set $Q\in\phi(p)$ such that $Q_\alpha\setminus Q$ is infinite for all $\alpha<\kappa$.

Let $P=\phi^{-1}(Q\cap (l\w+i))$ and for the neighborhood $\check B(P,g)$ of $\check p$ in $\check X$  find an ordinal $\alpha<\kappa$ such that $\check B(P_\alpha,f_\alpha)\subset\check B(P,f)$. By the choice of the set $Q$, the complement  $Q_\alpha\setminus Q$ is infinite. Then we can construct a sequence of points $(a_k)_{k\in\w}$ such that $\phi(a_k)\in Q_\alpha\setminus Q$ and $\phi(a_{k+1})>\phi(a_k)$ for every $k\in\w$.

The set $A=\{a_k\}_{k\in\w}$ is not bounded because it has infinite image $\phi(A)\subset\w$ under the bounded-to-bounded function $\phi$.

We claim that the sets $A$ and $B(P,f)$ are asymptotically disjoint.
This will follow as soon as we check that $$d(a_k,B(P,f))\ge f(a_k)=\tilde f\circ\phi(a_k).$$Assume conversely that $d(a_k,x)<f(a_k)$ for some $x\in B(P,f)$ and find a point $y\in P$ such that $x\in B(y,f)$. The choice of the function $f=\tilde f\circ\phi$ guarantees that $|\phi(a_k)-\phi(x)|\le \diam \,\phi(B(a_k,f))\le D$ and $|\phi(x)-\phi(y)|\le \diam \,\phi(B(y,f))\le D$.
Taking into account that $\phi(a_k)\in Q_\alpha\subset l\w+i$ and $\phi(y)\in \phi(P)\subset l\w+i$, we conclude that $\phi(a_k)-\phi(y)\in l\IZ$.
This fact combined with the upper bound
$$|\phi(a_k)-\phi(y)|\le|\phi(a_k)-\phi(x)|+|\phi(x)-\phi(y)|\le D+D<l$$implies that $\phi(a_k)=\phi(y)$, which is not possible as $\phi(y)\in Q$ and $\phi(a_k)\in Q_\alpha\setminus Q$.

This contradiction shows that the sets $A$ and $B(P,f)$ are asymptotically disjoint. Therefore, there exists $q\in A^\sharp$ such that $\check q\notin\check B(P,f)$ according to Lemma~\ref{l4}.
On the other hand, $A\subset P_\alpha\subset B(P_\alpha,f_\alpha)$ implies $\check q\in \check B(P_\alpha,f_\alpha)\subset\check B(P,f)$.
This contradiction completes the proof.
\end{proof}

\begin{lemma}\label{l7} If the space $X$ has no asymptotically isolated balls, then for each boundedly oscillating bounded-to-bounded function $\phi:X\to\w$ and each ultrafilter $p\in X^\sharp$ the point $\check p\in\check X$ has character $\chi(\check p,\check X)\ge \mathfrak q(\phi(p))$.
\end{lemma}

\begin{proof} Given any ultrafilter $p\in X^\sharp$, we need to check that $\chi(\check p)\ge\mathfrak q(\phi(p))$. To derive a contradiction, assume that the cardinal $\kappa=\chi(\check p)$ is smaller than $\mathfrak q(\phi(p))$.

Using Lemma~\ref{l3}, choose a transfinite sequence of pairs $\{(P_\alpha,f_\alpha)\}_{\alpha<\kappa}\subset p\times \w^{\upa X}$ such that for each $(P,f)\in p\times\w^{\upa X}$ there is $\alpha<\chi(\check p)$ such that $\check B(P_\alpha,f_\alpha)\subset \check B(P,f)$.

For every $\alpha<\kappa$ choose a bounded-to-bounded function $\tilde f_\alpha:\w\to\w$ such that $\tilde f_\alpha\circ\phi \le f_\alpha$. Such a function $\tilde f_\alpha$ can be defined by the formula  $\tilde f_\alpha(n)=\min f_\alpha \big(\phi^{-1}([n,\infty))\big)$ for $n\in\w$. Since $\kappa<\mathfrak q(\phi(p))=\coin(\w^{\upa\w},\le_{\phi(p)})$, there exists a non-decreasing function $\tilde f\in\w^{\upa \w}$ such that $\tilde f\le_{\phi(p)} \tilde f_\alpha$ for all $\alpha<\kappa$.

Since the function $\phi:X\to\w$ is boundedly oscillating and bounded-to-bounded we can replace $\tilde f$ by a smaller function, if necessary and assume additionally that
$$D=\sup_{x\in X}\diam\,\phi(B(x,\tilde f\circ\phi))<\infty,$$
see Lemma~\ref{l:bobf}. Let $f=\tilde f\circ\phi\in\w^{\upa X}$ and choose an integer number $l>3D$.

Since $X$ has no asymptotically isolated balls, there exists a non-decreasing function $\rho\in \w^{\upa\w}$ such that $\rho(n)\ge n$ and $B(x,\rho(n))\not\subset B(x,n)$ for all $n\in\w$ and $x\in X$. Let $n_0\ge D$ be an integer number such that $\tilde f(n_0)\ge 4\rho(0)$. For every $n<n_0$ put $g(n)=0$ and
for every $n\ge n_0$ let $\tilde g(n)$ be the largest number $m\in\w$ such that $\rho(6m)\le \frac14\tilde f(n)$.
In this way we define a non-decreasing bounded-to-bounded function $\tilde g:\w\to\w$ such that
$$6\tilde g(n)\le\rho(6\tilde g(n))\le\tfrac14\tilde f(n)\mbox{ \ for all \ $n\ge n_0$}.$$ The function $\tilde g$ induces a bounded-to-bounded function $g=\tilde g\circ\phi:X\to\w$.

For every $n\in\w$ using Zorn's Lemma, choose a maximal subset $S_n\subset \phi^{-1}(n)$, which is {\em $\tilde f(n)$-separated} in the sense that $d(x,y)\ge \tilde f(n)$ for any distinct points $x,y\in S_n$.

For every $i<l$, consider the set $X_i=\phi^{-1}(l\w+i)\subset X$ where $l\w+i=\{ln+i:n\in\w\}$. Divide each set $X_i$ into two subsets
$$B_i=X_i\cap\bigcup_{n\in l\w+i}B(S_n,2g)\mbox{ \ and \ }A_i=X_i\setminus B_i.$$

Since $p$ is an ultrafilter, there is a set $P\in p$ such that $P=A_i$ or $P=B_i$ for some $0\le i<l$.
By Lemma~\ref{l3}, the set $\check B(P,g)$ is a neighborhood of $\check p$ in $\check X$, so we can find an ordinal $\alpha<\kappa$ such that $\check B(P_\alpha,f_\alpha)\subset\check B(P,g)$.

By the choice of the function $\tilde f$, the set $\tilde Q_\alpha=\{n\in \w:\tilde f(n)\le\tilde  f_\alpha(n)\}$ belongs to the ultrafilter $\phi(p)$. Then the set $$Q_\alpha=P\cap P_\alpha\cap\phi^{-1}\big(\tilde Q_\alpha\cap (l\w+i)\big)$$ belongs to the ultrafilter $p$ and hence is unbounded.
This allows us to choose a sequence of points $(a_k)_{k\in\w}$ in $Q_\alpha$ such that $\phi(a_{k+1})>\phi(a_k)+2>n_0+2$ for every $k\in\w$.

Now we consider two cases.
\smallskip

1) $P=A_i$. For every $k\in\w$ the maximality of the $\tilde f(\phi(a_k))$-separated set $S_{\phi(a_k)}\subset \phi^{-1}(\phi(a_k))\subset X_i$ yields a point $s_k\in S_{\phi(a_k)}$ such that $d(a_k,s_k)<\tilde f(\phi(a_k))=f(a_k)$. Since $\phi(s_k)=\phi(a_k)\to\infty$, the set $\Sigma=\{s_k\}_{k\in\w}$ is unbounded and hence belongs to some ultrafilter $q\in X^\sharp$.

We claim that $\check q\in \check B(P_\alpha,f_\alpha)\setminus \check B(P,g)$, which will contradict the choice of $\alpha$.

To see that $\check q\in\check B(P_\alpha,f_\alpha)$, observe that for every $k\in\w$ we get $\phi(a_k)\in \tilde Q_\alpha$ and hence $\tilde f\circ\phi(a_k)\le \tilde f_\alpha\circ\phi(a_k)\le f_\alpha(a_k)$. This implies
$$s_k\in B(a_k,\tilde f\circ\tilde \phi(a_k))\subset B(a_k,f_\alpha)\subset B(P_\alpha,f_\alpha)$$and $\Sigma\subset B(P_\alpha,f_\alpha)$.

Lemma~\ref{l4} will imply that $\check q\notin \check B(P,g)$ as soon as we show that the sets $\Sigma=\{s_k\}_{k\in\w}$ and $B(P,g)$ are asymptotically disjoint.
 This will follow as soon as we check that
$d(s_k,B(P,g))\ge g(s_k)$ for every $k\in\w$. Assume conversely that $d(s_k,x)< g(s_k)$ for some $x\in B(P,g)$. Since $d(s_k,x)<g(s_k)=\tilde g\circ\phi(s_k)\le \tilde f\circ\phi(s_k)=f(s_k)$, the choice of the function $\tilde f$ guarantees that $|\phi(x)-\phi(s_k)|\le \diam\,\phi\big(B(s_k,f)\big)\le D$.

Since $x\in B(P,g)$, there is a point $y\in P$ with $d(x,y)\le g(y)$.
The inequality $d(x,y)\le g(y)=\tilde g\circ\phi(y)\le \tilde f\circ \phi(y)$ implies that $|\phi(x)-\phi(y)|\le l$.
It follows from $\phi(s_k)-\phi(y)\in (l\w+i)-(l\w+i)=l\IZ$ and
$$|\phi(s_k)-\phi(y)|\le |\phi(s_k)-\phi(x)|+|\phi(x)-\phi(y)|\le D+D<l$$that $\phi(s_k)=\phi(y)=n$ for some number $n\in\w$. Taking into account that $y\in P=A_i=X_i\setminus B_i\subset X_i\setminus B(s_k,2\tilde g(n))$, we conclude that $d(y,s_k)>2\tilde g(n)$ and hence $$d(x,s_k)\ge d(y,s_k)-d(x,y)>2\tilde g(n)- g(\phi(y))=2\tilde g(n)-\tilde g(n)=\tilde g(n)=g(s_k),$$which contradicts our assumption.
So, the sets $\Sigma$ and $B(P,g)$ are asymptotically disjoint and $\check q\notin \check B(P,g)$.
\smallskip

2) Now consider the second case $P=B_i$. By the choice of the function $\rho$, for every $k\in\w$ there is a point $b_k\in B(a_k,\rho(6g(a_k)))\setminus B(a_k,6g(a_k))$. Since
$d(b_k,a_k)\le \rho(6g(a_k))=\rho(6\tilde g\circ\phi(a_k))\le\tilde f\circ \phi(a_k)$, the choice of the number $D$ and the function $\tilde f$ guarantees that $|\phi(b_k)-\phi(a_k)|\le D$.
Since the sequence $(\phi(a_k))_{k\in\w}$ tends to infinity, so does the sequence $(\phi(b_k))_{k\in\w}$, which implies that the set $\Sigma=\{b_k\}_{k\in\w}$ is unbounded. So we can find an ultrafilter $q\in X^\sharp$ with $\Sigma\in q$.

We claim that $\check q\in \check B(P_\alpha,f_\alpha)$.
Indeed, for every $k\in\w$ we get $\phi(a_k)\in \tilde Q_\alpha$ and hence
$$b_k\in B\big(a_k,\rho(6g(a_k))\big)\subset B(a_k,\tilde f\circ \phi(a_k))\subset B(a_k,f_\alpha(a_k))\subset B(P_\alpha,f_\alpha).$$Consequently, $\Sigma\subset B(P_\alpha,f_\alpha)$ and $\check q\in\check B(P_\alpha,f_\alpha)$.

Next, we show that $\check q\notin \check B(P,g)$. By Lemma~\ref{l4}, it suffices to show that the sets $\Sigma$ and $B(P,g)$ are asymptotically disjoint. Since $\tilde g(\phi(b_k)-D)\to\infty$, this will follow as soon as we check that
$$d(b_k,B(P,g))\ge \tilde g(\phi(b_k)-D)\mbox{ for every $k\in\w$.}$$
Assuming the converse, find a point $x\in B(P,g)$ such that $d(b_k,x)<\tilde g(\phi(b_k)-D)$.

Since $$d(a_k,b_k)\le\rho(6\tilde g(\phi(a_k)))\le \tilde f\circ\phi(a_k),$$ the choice of the number $D$ guarantees that $|\phi(a_k)-\phi(b_k)|\le D$. Taking into account that
$a_k\in P=B_i$, find a point $s_k\in S_{\phi(a_k)}$ such that $a_k\in B(s_k,2 g)$ and $\phi(a_k)=\phi(s_k)\in l\w+i$.

Since $$d(b_k,x)<\tilde g(\phi(b_k)-D)\le\tilde g(\phi(b_k))\le \tilde f(\phi(b_k)),$$the choice of the number $D$ guarantees that $|\phi(b_k)-\phi(x)|\le \diam\,\phi(B(b_k,f))\le D$. Since $x\in B(P,g)$, there is a point $y\in P$ such that $x\in B(y,g)\subset B(y,f)$ and hence $|\phi(x)-\phi(y)|\le D$.
Since $y\in P=B_i$, there is a point $s\in S_{\phi(y)}$ such that $y\in B(s,2 g)$ and $\phi(s)=\phi(y)\in l\w+i$.

Taking into account that $\phi(s)-\phi(s_k)\in (l\w+i)-(l\w+i)=l\IZ$ and $$|\phi(s)-\phi(s_k)|\le|\phi(s)-\phi(y)|+|\phi(y)-\phi(x)|+|\phi(x)-\phi(b_k)|+|\phi(b_k)-\phi(a_k)|+|\phi(a_k)-\phi(s_k)|\le
0+D+D+D+0<l,$$ we conclude that $\phi(s)=\phi(s_k)$. Let $n=\phi(s)=\phi(s_k)=\phi(a_k)=\phi(y)$.

If $s=s_k$, then
$$
\begin{aligned}
d(b_k,x)&\ge d(b_k,a_k)-d(a_k,s_k)-d(s_k,s)-d(s,y)-d(x,y)\ge\\
&\ge 6g(a_k)-2g(s_k)-0-2g(s)-g(y)=6\tilde g(\phi(a_k))-2\tilde g(\phi(s_k))-2\tilde g(\phi(s))-g(\tilde (y))=\\
&=6\tilde g(n)-2\tilde g(n)-2\tilde g(n)-\tilde g(n)=\tilde g(n)=\tilde g(\phi(a_k))\ge \tilde g(\phi(b_k)-D),
\end{aligned}
$$which contradicts the choice of the point $x$.
\smallskip

If $s\ne s_k$, then $d(s,s_k)\ge \tilde f(n)$ by the choice of the $\tilde f(n)$-separated set $S_n$ and then
$$
\begin{aligned}
d(b_k,x)&\ge d(s_k,s)-d(s_k,a_k)-d(a_k,b_k)-d(x,y)-d(y,s)\ge\\
 &\ge \tilde f(n)-2 g(s_k)-\rho(6g(a_k))-g(y)-2g(s)=\\
 &=\tilde f(n)-2\tilde g(n)-\rho(6\tilde g(n))-\tilde g(n)-2\tilde g(n)=\\
 &=\tilde f(n)-\rho(6\tilde g(n))-6\tilde g(n)\ge \tilde f(n)-\rho(6\tilde g(n))-\rho(6\tilde g(n))\ge\\
 &\ge\tilde f(n)-2\rho(6\tilde g(n))\ge \tilde f(n)-\frac12\tilde f(n)=\frac12\tilde f(n)\ge \tilde g(n)=\tilde g(\phi(a_k))\ge \tilde g(\phi(b_k)-D).
 \end{aligned}
$$
Therefore $d(b_k,B(P,g))\ge \tilde g(\phi(b_k)-D)\to \infty$, which implies that the sets $B=\{b_k\}_{k\in\w}$ and $B(P,g)$ are asymptotically disjoint and $\check q\notin \check B(P,g)$.
\end{proof}

\begin{lemma}\label{l8} If an unbounded metric space $X$ has asymptotically isolated
balls, then its corona $\check X$ contains a closed-and-open subset,
homeomorphic to $\w^*$ and hence $\mchi(\check X)\le \mchi(\w^*)=\mathfrak
u$.
\end{lemma}

\begin{proof} Since $X$ has asymptotically isolated balls, there is $\e>0$
such that for each finite $\delta\ge\e$ there is an $\e$-ball $B_\e(x)$ equal
to the $\delta$-ball $B_\delta(x)$.
In particular, for the number $\delta_0=2\e$, we can find a point $x_0\in X$
such that  $B_\e(x_0)=B_{\delta_0}(x_0)$. By induction we shall construct an
increasing sequence of real numbers $(\delta_n)_{n=1}^\infty$ and a sequence
of points $(x_n)_{n\in\w}$ in $X$ the such that for every $n\in\IN$ the
following conditions are satisfied:
\begin{enumerate}
\item $\delta_n\ge (n+2)\e$;
\item $B_{\delta_n-\e}(x_k)\not\subset B_{2\e}(x_k)$ for all $k<n$;
\item $B_{\delta_n}(x_n)=B_\e(x_n)$.
\end{enumerate}
These conditions imply that for every $k<n$ we get $d_X(x_k,x_n)\ge
\delta_n$. Assuming the opposite, we get $x_k\in B_{\delta_n}(x_n)=B_\e(x_n)$
and hence $d_X(x_k,x_n)<\e$ and
$$B_{\delta_n-\e}(x_k)\subset B_{\delta_n}(x_n)=B_\e(x_n)\subset
B_{2\e}(x_k),$$
which contradicts the condition (2).

Consider the subspace $D=\{x_n\}_{n\in\w}\subset X$
and its $\e$-neighborhood
$$D_\e=\bigcup_{n\in\w}B_\e(x_n)=\bigcup_{n\in\w}B_{\delta_n}(x_n).$$
It follows that the characteristic function $f:X\to\{0,1\}$ of the set $D_\e$
is slowly oscillating.
It induces a continuous map $\check f:\check X\to\{0,1\}$ such that the
preimage $\check f^{-1}(1)$ is a clopen subset of $\check X$ that coincides
with the corona $\check D_\e$ of the set $D_\e$.

It is easy to check that the identity embedding $e:D\to D_\e$ is a coarse
equivalence, which induces a homeomorphism $\check e:\check D\to \check
D_\e$. Since each function on $D$ is slowly oscillating, the corona $\check
D$ of $D$ coincides with the Stone-\v Cech remainder $D^\sharp=\beta
D\setminus D$ of the discrete space $D$. Consequently, the corona $\check X$
contains a clopen subset $\check D_\e$, which is homeomorphic to
$\w^*=\beta\w\setminus\w$ and hence $\mchi(\check X)\le \mchi(\check
D)=\mchi(\w^*)=\mathfrak u$.
\end{proof}

Lemmas~\ref{l5}, \ref{l6}, \ref{l7} and \ref{l2} imply the following theorem, which is the main result of this section.

\begin{theorem}\label{t:char} Let $X$ be an unbounded metric space and $\phi:X\to\w$ be a boundedly oscillating bounded-to-bounded function. For each ultrafilter $p\in X^\sharp$ the point $\check p\in\check X$ has character
\begin{enumerate}
\item $\chi(\check p,\check X)\le\max\{\chi(p,X^\sharp),\mathfrak d\}$;
\item $\chi(\check p,\check X)\ge \chi(\phi(p),\w^*)\ge\mathfrak u$;
\item $\chi(\check p,\check X)\ge \max\{\chi(\phi(p),\w^*),\mathfrak q(\phi(p))\}\ge\max\{\mathfrak u,\mathfrak d\}$ if the space $X$ has no asymptotically isolated balls.
\end{enumerate}
\end{theorem}

\section{Proof of Theorem~\ref{main}}\label{s:pf-mainth}

We need to prove that for an unbounded metric space $X$ its corona $\check X$ has minimal character
\begin{itemize}
\item $\mchi(\check X)=\mathfrak u$ if $X$ has asymptotically isolated balls and
\item $\mchi(\check X)=\max\{\mathfrak u,\mathfrak d\}$, otherwise.
\end{itemize}

If $X$ has asymptotically isolated balls, then the corona $\check X$ has minimal character $\mchi(\check X)\le\mathfrak u$ by Lemma~\ref{l8}. The inequality $\mchi(\check X)\ge\mathfrak u$ follows from Theorem~\ref{t:char}(2).

If $X$ does not have asymptotically isolated balls, then $\mchi(\check X)\ge\max\{\mathfrak u,\mathfrak d\}$ by Theorem~\ref{t:char}(3). To prove the reverse inequality, take any injective function $f:\w\to X$ such that $\lim_{n\to\infty}d(f(n),f(0))=\infty$. Choose any ultrafilter $\U\in\w^*$ with $\chi(\U,\w^*)=\mathfrak u$ and consider its image $p=\beta f(\U)\in \beta X$. The choice of the function $f$ guarantees that $p\in X^\sharp$. It follows that $\chi(p,X^\sharp)=\chi(\U,\w^*)=\mathfrak u$ and then $$\mchi(\check X)\le\chi(\check p,\check X)\le \max\{\chi(p,X^\sharp),\mathfrak d\}=\max\{\mathfrak u,\mathfrak d\}$$
according to Theorem~\ref{t:char}(1).

\section{Proof of Theorem~\ref{t2}}\label{s2}

It is easy to see that the Cantor macro-cube $C=2^{<\IN}$ has no
asymptotically isolated balls.  Consequently, $\mchi(\check
C)=\max\{\mathfrak u,\mathfrak d\}=\mathfrak d$ by Theorem~\ref{main}. By
\cite{Dra}, $\dim(\check C)=\asdim(C)=0$. Now we are ready to prove the
implications $(1)\Ra(2)\Ra(3)\Ra(1)$ of Theorem~\ref{t2}. Let $(X,d_X)$ be a
metric space of bounded geometry.

$(1)\Ra(2)$. If $X$ is coarsely homeomorphic to the Cantor macro-cube
$C=2^{<\IN}$, then the coronas of $X$ and $C$ are homeomorphic according to
\cite[2.42]{Roe}.

$(2)\Ra(3)$ If the coronas $\check X$ and $\check C$ are homeomorphic, then
$\dim(\check X)=\dim(\check C)=\asdim(C)=0$ and $\mchi(\check X)=\mchi(\check
C)=\mathfrak d$.

$(3)\Ra(1)$ Assume that $\dim(\check X)=0$ and $\mchi(\check X)=\mathfrak d>\mathfrak u$. By Proposition~\ref{t:dim} and Theorem~\ref{main}(1), the metric space $X$ has asymptotic dimension zero and has no asymptotically isolated balls.  Since $X$ has bounded geometry, the characterization
theorem \cite{BZ} implies that the metric space $X$ is coarsely equivalent to the Cantor macro-cube $2^{<\IN}$.


\begin{thebibliography}{}

\bibitem{BZ} T.Banakh, I.Zarichnyi, {\em Characterizing the Cantor bi-cube in
asymptotic categories}, Groups, Geometry, and Dynamics, {\bf 5}:4 (2011) 691--728.

\bibitem{BZ2} T.Banakh, I.Zarichnyi, {\em A coarse characterization of the
Baire macro-space}, Proc. of Intern. Geometry Center, {\bf 3}:4 (2010) 6--14.

\bibitem{BZsurv} T.~Banakh, L.~Zdomskyy, {\em The coherence of semifilters: a survey. Selection principles and covering properties in topology}, 53--105, Quad. Mat., 18, Dept. Math., Seconda Univ. Napoli, Caserta, 2006.

\bibitem{BZbook}  T. Banakh,  L. Zdomskyy,  {\it Coherence of Semifilters}, book in progress.\\
\texttt{http://www.franko.lviv.ua/faculty/mechmat/Departments/Topology/booksite.html}.

\bibitem{BD} G.~Bell, A.~Dranishnikov, {\em Asymptotic dimension}, Topology
Appl. {\bf 155}:12 (2008) 1265--1296.

\bibitem{Bla86} A. Blass, \emph{Near coherence of filters, I. Cofinal equivalenceof models of arithmetic,} Notre Dame Journal of Formal Logic, \textbf{27} (1986)  579--591.

\bibitem{Blass} A.~Blass, {\em Combinatorial Cardinal
Characteristics of the Continuum}, in: Handbook of Set Theory, Ch.6, pp.395--489, Springer, Dordrecht, 2010.

\bibitem{Canjar} M.~Canjar, {\em Cofinalities of countable ultraproducts: the existence theorem},
Notre Dame J. Formal Logic {\bf 30}:4 (1989) 539--542.

\bibitem{vD} E.~van Douwen, {\em The integers and topology}, in: Handbook of
set-theoretic topology, 111--167, North-Holland, Amsterdam, 1984.

\bibitem{Dra} A. Dranishnikov, {\em Asymptotic topology}, Uspekhi Mat. Nauk
{\bf 55}:6 (2000) 71--116.

\bibitem{DKU} A.N.~Dranishnikov, J.~Keesling, V.V.~Uspenskij, {\em
On the Higson corona of uniformly contractible spaces}, Topology {\bf 37}:4 (1998) 791--803.

\bibitem{DZ} A.~Dranishnikov, M.~Zarichnyi, {\em Universal spaces for
asymptotic dimension},
Topology Appl. {\bf 140}:2-3 (2004) 203--225.

\bibitem{Fremlin} D.~Fremlin, {\em Consequences of Martin's Axiom}, Cambridge
Tracts in Mathematics. V.{\bf 84}, Cambridge University Press, London,
1984.

\bibitem{Ke} A.~Kechris, {\em Classical Descriptive Set Theory}, Springer-
Verlag, New York, 1995.

\bibitem{LafZhu98}  C. Laflamme, J.-P. Zhu,
{\it The Rudin-Blass ordering of ultrafilters,}
J. Symbolic Logic  \textbf{63}  (1998)   584--592.

\bibitem{Prot03} I.V.~Protasov, {\em Normal ball structures}, Mat. Stud. {\bf 20} (2003) 3--16.

\bibitem{Prot05} I.V.~Protasov, {\em Coronas of balleans}, Topology Appl. {\bf 149}:1-3
(2005) 149--160.

\bibitem{Prot11} I.V.Protasov, {\em Coronas of ultrametric spaces}, Coment.
Math. Univ. Carolin. {\bf 52} (2011) 303--307.

\bibitem{Roe} J.Roe, {\em Lectures on coarse geometry}, Amer. Math. Soc.,
Providence, RI, 2003.

\bibitem{Vau} J.~Vaughan, {\em Small uncountable cardinals and topology}, in:
Open Problems in Topology, 195--218, North-Holland, Amsterdam, 1990.


\end{thebibliography}
\end{document}